\documentclass[12pt]{amsart}
\pdfoutput=1
\title{ F-threshold of determinantal rings }
\author{Barbara Betti, Alessio Moscariello, Francesco Romeo, Jyoti Singh}
\address{(Barbara Betti) Max Planck Institute for Mathematics in the Sciences, Inselstrasse 22, 04103
Leipzig, Germany}
\email{barbara.betti@mis.mpg.de}
\address{(Alessio Moscariello) Dipartimento di Matematica e Informatica\\Università degli Studi di Catania\\Catania\\Italy}
\email{alessio.moscariello@unict.it}
\address{(Francesco Romeo) Department of Mathematics \\ University of Trento \\ via Sommarive, 14, 38123 Povo (Trento)\\ Italy }
\email{francesco.romeo-3@unitn.it}
\address{(Jyoti Singh) Department of Mathematics, Visvesvaraya National Institute of Technology, Nagpur\\ India}
\email{jyotisingh@mth.vnit.ac.in}
\date{}

\usepackage[utf8]{inputenc}
\usepackage[T1]{fontenc}

\usepackage{geometry}
\usepackage{blindtext}
\usepackage{amsmath}
\usepackage{amsthm}
\usepackage{dsfont}
\usepackage{amssymb}
\usepackage[shortlabels]{enumitem}
\usepackage{booktabs}
\usepackage{color}
\usepackage{graphicx}
\graphicspath{{./images/}}
\usepackage[dvipsnames]{xcolor}
\usepackage{tikz}
\usepackage{tikz-cd}
\usepackage{bm}
\usepackage{bbm}
\usepackage{nicefrac}
\usepackage{mathtools}
\usepackage{epstopdf}
\usepackage{multicol}
\usepackage{rotating}
\usepackage[shortcuts]{extdash}
\usepackage{hyperref}
\usepackage{tikz,pgf}
\usetikzlibrary{arrows}
\usepackage{xspace}
\usepackage[all]{xy}
\usepackage{mathrsfs}
\usepackage{verbatim}
\usepackage{marvosym}

\usepackage{alphabeta}

\usepackage[greek,english]{babel}

\usepackage{listings}
\usepackage{hyperref}
\hypersetup{
	pdfpagemode=UseNone,
	colorlinks=true,
	linkcolor=blue,
	citecolor=blue,
	urlcolor=blue,
	allbordercolors=white,
	pdftitle={Pragmatic Open Problem n.2},
	pdfauthor={Barbara Betti}}

\newtheorem{thm}{Theorem}
\newtheorem{lemm}[thm]{Lemma}
\newtheorem{prop}[thm]{Proposition}
\newtheorem{cor}[thm]{Corollary}

\theoremstyle{definition}
\newtheorem{examplex}[thm]{Example}
\newtheorem{problem}[thm]{Problem}
\newtheorem{conjecture}[thm]{Conjecture}
\newenvironment{example}
{\pushQED{\qed}\examplex}
{\popQED\endexamplex}

\newtheorem{defn}[thm]{Definition}

\theoremstyle{remark}
\newtheorem{remark}[thm]{Remark}

\newcommand{\numberset}{\mathbb}

\newcommand{\D}{\mathcal{D}}

\newcommand{\N}{\numberset{N}}

\definecolor{cof}{RGB}{219,144,71}
\definecolor{pur}{RGB}{186,146,162}
\definecolor{greeo}{RGB}{91,173,69}
\definecolor{greet}{RGB}{52,111,72}

\newcommand{\mm}{\mathfrak{m}}

\begin{document}
	
	\maketitle
	
	\begin{abstract}
		In this paper, by using a combinatorial approach, we establish a new upper bound for the F-threshold $c^\mm(\mm)$ of determinantal rings generated by maximal minors. We prove that $c^\mm(\mm)$ coincides with the $a$-invariant in the case of $3\times n$ and $4\times n$ matrices and we conjecture such equality holds for all matrices.
	\end{abstract}

\section{Introduction}
The $F$-threshold is an important numerical invariant of positive characteristic rings that measure the behaviour of singularities. It was first introduced by Mustaţă, Takagi and Watanabe for regular local rings \cite{Musta2004FthresholdsAB} as analogue of log-canonical threshold for characteristic $p$ rings \cite{TAKAGI2004278}. In \cite{HMTW2008}, Huneke, Mustaţă, Takagi and Watanabe extended this invariant to a more general setting and related it to tight closure, integral closure and multiplicities.  A formula of the $F$-threshold for binomial hypersurfaces is given by Matsuda, Ohtani and Yoshida in \cite{doi:10.1080/00927870903107940}. Further, Chiba and Matsuda computed the F-threshold of Hibi rings, which are a kind of graded toric rings on a finite distributive lattice in \cite{CM2015}. The existence of $F$-threshold for every noetherian ring was proved by De Stefani, Nuñez-Betancourt and Pérez \cite{destefani2017existence}. However, the actual computation of the $F$-threshold for large classes of rings is still a challenging problem. In this paper, we approach this problem for a subclass of determinantal rings. 

Let $K$ be a $F$-finite field of characteristic $p>0$, $X=(x_{ij})$ a $n\times m$ matrix of indeterminates and $S=K[X]$ the polynomial ring over the entries of $X$. Given $t\leq m\leq n$ we consider the ideal $I_t$ of the $t\times t$ minors of $X$ and the corresponding determinantal ring $R=S/I_t$. These rings and their properties are extensively studied since they are ubiquitous in commutative algebra and algebraic geometry. It is indeed well-known that determinantal rings are Cohen-Macaulay and, if $n=m$, also Gorenstein (\cite[Thm 5.5.6]{SVANES1974369}). In positive characteristic determinantal rings are strongly $F$-regular; this was firstly proved by Hochster and Huneke in \cite{hochster1994tight} and by Bruns and Conca in \cite[Thm 3.1]{10.1307/mmj/1030132183} with a different approach. Now we recall the definition of $F$-threshold of a ring.

\begin{defn}
    Let us consider $R$ a ring of characteristic $p>0$ and $\mathfrak{a},J\subseteq R$ two ideals such that $\mathfrak{a}\subseteq \sqrt{J}$. For a fixed positive integer $e$ we define the finite integer
    \[ \nu_{\mathfrak{a}}^J(p^e)\colon= \max \{r\in \N \mid \mathfrak{a}^r \not\subseteq J^{[p^e]} \}. \]
    The \emph{$F$-threshold of $\mathfrak{a}$ with respect to $J$} is the limit
    \[ c^J(\mathfrak{a})\colon=  \lim_{e\to\infty} \dfrac{\nu_{\mathfrak{a}}^J(p^e)}{p^e}.\]
\end{defn}

If $R$ is a graded ring with maximal homogeneous ideal $\mathfrak{m}$, we refer to $c^{\mathfrak{m}}(\mathfrak{m})$ as the \emph{$F$-threshold of $R$}. In the literature this invariant is also called \emph{diagonal $F$-threshold}, as in \cite{doi:10.1080/00927872.2013.772187}. As mentioned before, we consider the following problem.

\begin{problem}\label{prob}
    Let $K$ be a $F$-finite field, $S=K[X]$, where $X=(x_{ij})$ is a $n\times m$ matrix of variables and $I=I_t(X)$ be the ideal generated by $t \times t$ minors of $X$, with $t\leq m\leq n$. Compute the $F-$threshold $c^{\mm}(\mm)$ of $R=S/I$, where $\mm \subseteq R$ is the maximal homogeneous ideal.
\end{problem}

The first unknown case is $t=m=n=3$, while the $t=m=n=2$ is solved in \cite[Thm 2]{doi:10.1080/00927870903107940}, for those parameters we have $c^{\mathfrak{m}}(\mathfrak{m})=2$. A computational approach to calculate the $F$-threshold for specific values of $n,m,t$ is not possible. Indeed computing $\nu^\mathfrak{m}_\mathfrak{m}(p^e)$, for example using the \texttt{Macaulay2} package \texttt{FrobeniusThresholds} \cite{FrobeniusThresholdsSource}, is unfeasible except for small values of $e$. However, some bounds are known.

\begin{thm}\label{lowerbound} Let $R$ be a standard graded $K$-algebra that is F-finite and F-pure, then:
\[ {\rm fpt}(\mathfrak{m})\leq -a(R)\leq c^{\mm}(\mm)\leq {\rm dim}(R). \]
    In the context of Problem \ref{prob}, we have
    \[  m(t-1)\leq n(t-1)  \leq c^{\mm}(\mm)\leq (m+n-t+1)(t-1).  \] 
\end{thm}

\noindent
The invariant ${\rm fpt}(\mathfrak{m})$ is the \emph{$F$-pure threshold} of $R$, introduced in \cite{TAKAGI2004278}. Its computation for determinantal rings is given by Singh, Takagi and Varbaro in \cite[Prop. 4.3]{singh2016gorenstein}. The chain of inequalities, where $a(R)$ is the \emph{a-invariant} of $R$, are proved in \cite[Thm B]{article} and \cite[Prop 1.7]{Musta2004FthresholdsAB}.

In the following we state our main results. We focus on the case of maximal minors, that is the case $t=m$ of Problem \ref{prob}. We give a new sharper upper bound for the $F$-threshold and we prove that it is equal to the $a$-invariant in the cases $t=m\in\{3,4\}$.

\begin{thm}\label{theorem upper bound}
    Let $R=S/I_t(X)$, where $X$ is a $n\times m$ matrix of indeterminates and $t=m\leq n$. Then for $n\geq 3$ we have
    \[  c^{\mm}(\mm)\leq  n\left(t-\frac{1}{2}\right)  .
  \]
   
\end{thm}

\begin{thm}\label{thm n=3,4}
    Let $R=S/I_t(X)$, where $X$ is a $n\times m$ matrix of indeterminates and $t=m\leq n$. Then for $t=3,4$ we have  
    \[    c^{\mm}(\mm)=n(t-1). \]
\end{thm}
 Both proofs of Theorem \ref{theorem upper bound} and \ref{thm n=3,4} rely on the case $t=m=n$. This case is solved using a combinatorial approach to compute $ \nu_{\mathfrak{m}}^{\mathfrak{m}}(p^e)$ in a way that is independent of the characteristic of the field. For values of $t$ greater than $4$ we conjecture that the same formula as in Theorem \ref{thm n=3,4} holds.

\begin{conjecture} \label{conj F-threshold}
    In the context of Theorem \ref{theorem upper bound}, $c^{\mm}(\mm)=n(t-1)$ for all $t,n\geq 5$.
\end{conjecture}

This Conjecture is based on the assumption that $\nu^\mathfrak{m}_\mathfrak{m}(p^e) = t(t-1)(p^e-1)$ if $t=m=n$, which we proved for $t=3,4$ to get Theorem \ref{thm n=3,4}. 

The structure of the paper is as follows: in Section 2 we state the problem and explain the general context of our approach. Then, in Section 3 we describe the combinatorial framework and give the main results for the square case $t=m=n$. Finally, using these results, in Section 4 we prove Theorem \ref{theorem upper bound} and \ref{thm n=3,4}.

\section{Set up} 

As mentioned before, our approach does not rely on any assumption on the base field $K$. In particular, the characteristic $p$ does not play any role in the results presented here. Indeed, we study the following, more general problem.

\begin{problem}\label{altprob}
	Let $K$ be a field, $S=K[X]$, where
	\[  X=\begin{pmatrix}
		x_{11} & \dots & x_{1m} \\
		\vdots & \ddots & \vdots \\
		x_{n1} & \dots & x_{nm}
	\end{pmatrix}\]
	is a $n\times m$ matrix of variables, and let $I=I_t(X)$ be the ideal generated by $t \times t$ minors of $X$, with $t \le m \le n$. For a fixed positive integer $k \in \mathbb{Z}^+$ consider the generalized Frobenius power $\mathfrak{m}^{[k]}=(x_{11}^k,\dots, x_{nm}^k)$, and compute
	\[ \nu_k(\mm)={\rm max}\{r\in \N | \mm^r\not\subseteq (\mm^{[k]},I) \}. \]
\end{problem}

\noindent
Setting $K$ a field of characteristic $p$, we retrieve $ \nu_{\mm}^{\mm}(p^e)=\nu_{p^e}(\mm)$ and the $F$-threshold of $R=S/I$ as $\displaystyle c^\mm(\mm)=\lim_{e\to\infty} \dfrac{\nu_{p^e}(\mm)}{p^e}$.

Moreover, since $\nu_k(\mm)$ is equal to the maximum degree $r$ of a monomial $g \in K[X]$ not belonging to $(\mm^{[k]},I)$, we focus on finding conditions ensuring that $g \in (\mm^{[k]},I)$. In order to do so, for a monomial $g=x_{11}^{\alpha_{11}}\ldots x_{nm}^{\alpha_{nm}}$ we consider the \textit{exponent matrix}

\[  M(g) :=\begin{pmatrix}
	\alpha_{11} & \dots & \alpha_{1m} \\
	\vdots & \ddots & \vdots \\
	\alpha_{n1} & \dots & \alpha_{nm}
\end{pmatrix}.\]
Our goal is to study $M(g)$ and find conditions on its entries, which would in turn ensure that the associated monomial $g$ belongs to the ideal $(\mm^{[k]},I)$. However, since these conditions are not satisfied by all monomials of the degree required to compute $\nu_k(\mm)$, we also need to work modulo $I$ by finding a suitable representation of $g$ as a finite sum of monomials, which will then be the object of our study.

The main combinatorial framework is defined under the assumption $t=m=n$ and it is the main tool of the proofs. The results obtained are then extended to the case $t = m \le n$ with a counting argument.

\section{Determinants of square matrices} 

If $t=m=n$, $I_t(X)$ is a principal ideal generated by $f:={\rm det}(X)$, therefore, in order to understand when a monomial $g \in K[x_{11},\ldots,x_{nn}]$ belongs to $(\mm^{[k]},f)$ we need to work modulo $f$. Our intuitive approach is the following: we choose one monomial appearing in $f$ that divides $g$ and substitute it with all the other monomials appearing in $f$ (with the appropriate sign), obtaining a sum of monomials which is equivalent to $g$ modulo $f$. Our aim is to get a "special" representation in this way; for this purpose, we need some auxiliary definitions.

\begin{defn} \label{def diagonal}
	We say that $d\subseteq \{1,\dots,n\} \times \{1,\ldots,n\}$ is a \textit{diagonal} if the elements of $d$ are the indexes of one of the monomials appearing in ${\rm det}(X)$ (then $|d|=n$). We denote the set of all diagonals by $\D$.
\end{defn}

While diagonals are defined in terms of monomials of ${\rm det}(X)$, they can also be viewed in terms of $M(g)$: in fact, diagonals are sets of $n$ elements of $M(g)$ containing exactly one element from each row and column of $M(g)$.

Given a set $A$, the \textit{indicator function} $\mathds{1}_A(x)$ is equal to $1$ if $x \in A$ and $0$ otherwise.
\begin{defn}\label{def swap}
Let $g,g' \in K[x_{11},\ldots,x_{nn}]$ be monomials and let $d,d'$ be diagonals. We say that $g'=x_{11}^{\beta_{11}}\ldots x_{nn}^{\beta_{nn}}$ is obtained from $g=x_{11}^{\alpha_{11}}\ldots x_{nn}^{\alpha_{nn}}$ by a \textit{swap} of $d$ with $d'$ if, for every $i,j \in \{1,\ldots,n\}$, $\beta_{ij}=\alpha_{ij}-\mathds{1}_d(i,j)+\mathds{1}_{d'}(i,j)$.	
\end{defn}

Definition \ref{def swap} has an intuitive visualization in terms of the exponent matrix, since $M(g')$ is obtained from $M(g)$ by reducing by one all entries $\alpha_{ij}$ such that $(i,j) \in d$ and increasing by one all entries such that $(i,j) \in d'$.
The following Proposition enables us to work modulo $f$ and consider a special class of monomials satisfying some useful conditions.

\begin{prop}\label{repexists}
	Let $g \in K[x_{11},\ldots,x_{nn}]$ be a monomial and let $d \in \D$ be a diagonal.
	Then there exists a finite set of (not necessarily distinct) monomials $g_1,\ldots,g_h \in K[x_{11},\ldots,x_{nn}]$, considered with sign, such that:
	\begin{enumerate}
		\item $\deg g_i = \deg g$ for every $i=1,\ldots,h$.
		\item For every $g_i=x_{11}^{\beta_{11}}\ldots x_{nn}^{\beta_{nn}}$ we have $\displaystyle \prod_{(r,q) \in d} \beta_{rq} = 0$;
		\item Every monomial $g_i$ is obtained from $g$ by a chain of swaps of $d$ with other diagonals; 
		\item $g \equiv g_1 + \ldots + g_h \pmod{f}$.
	\end{enumerate}
\end{prop}

\begin{proof}
	
	For a monomial $g = x_{11}^{\alpha_{11}}\ldots x_{nn}^{\alpha_{nn}}$ define the quantity $\displaystyle \Sigma_{d}(g) = \sum_{(r,q) \in d} \alpha_{rq}$; we proceed by induction on $\Sigma_d(g)$. If $\Sigma_d(g) \in \{0,1,\dots,n-1\}$, since $|d|=n$ there exists $(r,q) \in d$ such that $\alpha_{rq} = 0$: therefore, $g$ itself satisfies the conditions listed in the thesis, and thus the claim is proved.
	
	Assume then that our thesis is true for all monomials $g'$ such that $\Sigma_d(g') < h$, and consider a monomial $g = x_{11}^{\alpha_{11}}\ldots x_{nn}^{\alpha_{nn}}$ such that $\Sigma_d(g) = h$. Like before, if there exists $(r,q) \in d$ such that $\alpha_{rq} = 0$ then $g$ itself satisfies the conditions in the thesis; assume therefore that $\displaystyle \prod_{(r,q) \in d} \alpha_{rq} > 0$. Then $\displaystyle \prod_{(r,q) \in d} x_{rq}$ divides $g$, and thus, working modulo $f$, we can replace the monomial $\displaystyle \prod_{(r,q) \in d} x_{rq}$ in $g$ with all other $n!-1$ monomials appearing in $f$. Thus we obtain an equivalence
	\[ g \equiv g_1+\dots+g_{n!-1} \pmod{f},\] 
	where each $g_i$ is considered with the appropriate sign.
	By definition of diagonal every other monomial appearing in $f$ is of the form  $\{x_{rq}\}_{(r,q) \in d'}$ for some diagonal $d' \neq d$: therefore, every monomial $g_i$ is obtained by a swap of $d$ with another diagonal $d'$, and  $\displaystyle \Sigma_d(g_i) = \Sigma_d(g) - n + \sum_{(r,q) \in d} \mathds{1}_{d'}(r,q) <  h$ since $d'$ cannot contain $d$. Then every $g_i$ satisfies the induction hypothesis, and is thus equivalent modulo $f$ to a linear combination of monomials satisfying the conditions listed. Finally, since every $g_i$ is obtained by a swap of $d$ with another diagonal $d'$, $g$ can be written as a linear combination of monomials satisfying the conditions listed. 
\end{proof}

\begin{defn} \label{def.d_representation_n}
	Let $g \in K[x_{11},\ldots,x_{nn}]$ be a monomial, and let $d \in \D$ be a diagonal. A \textbf{$d$-represenation} for $g$ is an equivalence of the form $g \equiv g_1 + \ldots + g_h \pmod{f}$, where $g_1,\ldots,g_h \in K[x_{11},\ldots,x_{nn}]$ are monomials (considered with sign) such that, for every $i=1,\ldots,h$, $\deg g_i = \deg g$, $g_i=x_{11}^{\beta_{11}}\ldots x_{nn}^{\beta_{nn}}$ is obtained by a chain of swaps of $d$ with other diagonals, and satisfies $\displaystyle \prod_{(r,q) \in d} \beta_{rq} = 0$.
\end{defn}
 
\begin{example}\label{example}
	Let $n=3$, and consider $g=x_{11}^7 x_{12}^4 x_{13}^3 x_{21}^5 x_{22}^2 x_{23} x_{31}^3 x_{32}^2 x_{33}^2$ and $d=\{(1,1),(2,2),(3,3)\}$. Using the algorithm described in the proof of Proposition \ref{repexists} we compute a $d$-representation for $g$. The following list, computed with GAP (\cite{GAP2022}), describes the exponent matrices of all monomials appearing in the desired $d$-representation.
	
	\begin{scriptsize}

		\begin{verbatim}
		[ [ [ 7, 4, 3 ], [ 5, 0, 3 ], [ 3, 4, 0 ] ], [ [ 6, 5, 3 ], [ 6, 0, 2 ], [ 3, 3, 1 ] ],
		[ [ 6, 5, 3 ], [ 5, 0, 3 ], [ 4, 3, 0 ] ], [ [ 6, 4, 4 ], [ 6, 0, 2 ], [ 3, 4, 0 ] ],
		[ [ 6, 4, 4 ], [ 5, 1, 2 ], [ 4, 3, 0 ] ], [ [ 6, 5, 3 ], [ 6, 0, 2 ], [ 3, 3, 1 ] ],
		[ [ 5, 6, 3 ], [ 7, 0, 1 ], [ 3, 2, 2 ] ], [ [ 5, 6, 3 ], [ 6, 0, 2 ], [ 4, 2, 1 ] ],
		[ [ 5, 5, 4 ], [ 7, 0, 1 ], [ 3, 3, 1 ] ], [ [ 5, 5, 4 ], [ 6, 0, 2 ], [ 4, 3, 0 ] ],
		[ [ 4, 6, 4 ], [ 7, 0, 1 ], [ 4, 2, 1 ] ], [ [ 4, 6, 4 ], [ 6, 0, 2 ], [ 5, 2, 0 ] ],
		[ [ 4, 5, 5 ], [ 7, 0, 1 ], [ 4, 3, 0 ] ], [ [ 4, 5, 5 ], [ 6, 1, 1 ], [ 5, 2, 0 ] ],
		[ [ 6, 5, 3 ], [ 5, 0, 3 ], [ 4, 3, 0 ] ], [ [ 5, 6, 3 ], [ 6, 0, 2 ], [ 4, 2, 1 ] ],
		[ [ 5, 6, 3 ], [ 5, 0, 3 ], [ 5, 2, 0 ] ], [ [ 5, 5, 4 ], [ 6, 0, 2 ], [ 4, 3, 0 ] ],
		[ [ 5, 5, 4 ], [ 5, 1, 2 ], [ 5, 2, 0 ] ], [ [ 6, 4, 4 ], [ 6, 0, 2 ], [ 3, 4, 0 ] ],
		[ [ 5, 5, 4 ], [ 7, 0, 1 ], [ 3, 3, 1 ] ], [ [ 5, 5, 4 ], [ 6, 0, 2 ], [ 4, 3, 0 ] ],
		[ [ 5, 4, 5 ], [ 7, 0, 1 ], [ 3, 4, 0 ] ], [ [ 5, 4, 5 ], [ 6, 1, 1 ], [ 4, 3, 0 ] ],
		[ [ 6, 4, 4 ], [ 5, 1, 2 ], [ 4, 3, 0 ] ], [ [ 5, 5, 4 ], [ 6, 0, 2 ], [ 4, 3, 0 ] ],
		[ [ 4, 6, 4 ], [ 7, 0, 1 ], [ 4, 2, 1 ] ], [ [ 4, 6, 4 ], [ 6, 0, 2 ], [ 5, 2, 0 ] ],
		[ [ 4, 5, 5 ], [ 7, 0, 1 ], [ 4, 3, 0 ] ], [ [ 4, 5, 5 ], [ 6, 1, 1 ], [ 5, 2, 0 ] ],
		[ [ 5, 5, 4 ], [ 5, 1, 2 ], [ 5, 2, 0 ] ], [ [ 5, 4, 5 ], [ 6, 1, 1 ], [ 4, 3, 0 ] ],
		[ [ 5, 4, 5 ], [ 5, 2, 1 ], [ 5, 2, 0 ] ] ]
		\end{verbatim}
	\end{scriptsize}
	Notice that the same monomial appears more than once in this $d$-representation of $g$ (and since each monomial is considered with sign, it is indeed possible that there are some cancellations). For instance, consider the monomial $g'=x_1^5x_2^5x_3^4x_4^5x_5x_6^2x_7^5x_8^2$, which exponent matrix $\begin{pmatrix} 5 & 5 & 4 \\ 5 & 1 & 2 \\ 5 & 2 & 0 \end{pmatrix}$ appears twice in the list above. In fact, this matrix can be obtained from $M(g)$ with the chain of swaps (where in every step entries in $d$ are decreased by one, while the diagonal highlighted in red is the one which has been increased by one)
	
	$$\begin{pmatrix} 7 & 4 & 3 \\ 5 & 2 & 1 \\ 3 & 2 & 2 \end{pmatrix} \rightarrow \begin{pmatrix} 6 & \color{red}{5} & 3 \\ 5 & 1 & \color{red}{2} \\ \color{red}{4} & 2 & 1 \end{pmatrix} \rightarrow \begin{pmatrix} 5 & 5 & \color{red}{4} \\ 5 & \color{red}{1} & 2 \\ \color{red}{5} & 2 & 0 \end{pmatrix}$$ 
	
		and 
	
	$$\begin{pmatrix} 7 & 4 & 3 \\ 5 & 2 & 1 \\ 3 & 2 & 2 \end{pmatrix} \rightarrow \begin{pmatrix} 6 & 4 & \color{red}{4} \\ 5 & \color{red}{2} & 1 \\ \color{red}{4} & 2 & 1 \end{pmatrix} \rightarrow \begin{pmatrix} 5 & \color{red}{5} & 4 \\ 5 & 1 & \color{red}{2} \\ \color{red}{5} & 2 & 0 \end{pmatrix}.$$
\end{example}

Despite the fact that there may be some repetitions in $d$-representations, as the previous Example shows, we choose to discard all signs and consider these repetitions as distinct monomials. Our aim is to prove that, under certain conditions, every monomial appearing in a $d$-representation, regardless of cancellation, belongs to $(\mm^{[k]},f)$. This is done by studying some specific sets of entries of $M(g)$. 
\begin{defn}\label{def minimalsets}
	Let $M$ be a $n \times n$ matrix and $d \in \D$ a diagonal. We say that a set $S \subseteq \{1,\ldots,n\} \times \{1,\ldots,n\}$ has \textit{minimal intesection with $d$} if for any other diagonal $d'$ we have $|S \cap d| \le |S \cap d'|$.
\end{defn}

While diagonals and sets with minimal intersections are defined as sets of indices, we will sometimes identify them by their entries on specific matrices, for a more immediate visualization.
\begin{remark}\label{minimalsets}
	\begin{enumerate}
		\item Since each row and column intersects every diagonal exactly once, rows and columns have minimal intersection with every diagonal. 
		\item For a fixed diagonal $d$, minors $T$ of size $n - 1$ such that $|T \cap d | = n-2$ have minimal intersection with $d$, as clearly from the definition of diagonal we have $|T \cap d'| \ge n-2$ for every $d' \in \D$.
	\end{enumerate}
\end{remark}

\begin{defn} \label{kprespecial_n=m=t}
	Let $k$ be a positive integer, $d \in \D$ a diagonal, and $g = x_{11}^{\alpha_{11}}\ldots x_{nn}^{\alpha_{nn}}$ a monomial. We say that $d$ is \textbf{$k$-special} for $g$ if, for every $(i,j) \in d$ there exists a set $S \subseteq \{1,\ldots,n\} \times \{1,\ldots,n\}$ containing $(i,j)$ and having minimal intersection with $d$, such that $$\sum_{(r,q) \in S} \alpha_{rq} \ge (|S|-1)(k-1)+1$$ 
\end{defn}

The following Lemmas will allow us to translate the $k$-special property to properties of $d$-representations.
\begin{lemm}\label{sumswap}
	Let $g,g' \in K[x_{11},\ldots,x_{nn}]$ be monomials, and let $M(g)=(\alpha_{ij})$ and $M(g')=(\beta_{ij})$ be their exponent matrices. Let $d \in \D$ be a diagonal, and $S$ a set having minimal intersection with $d$. Assume that $g'$ is obtained from $g$ by a swap of $d$ with another diagonal $d'$. Then
	
	$$\sum_{(r,q) \in S} \alpha_{rq} \le \sum_{(r,q) \in S} \beta_{rq},$$
	
	and equality holds if $S$ has minimal intersection with both $d$ and $d'$.
\end{lemm}

\begin{proof}
	
By definition of swap we have $\beta_{ij}=\alpha_{ij}-\mathds{1}_d(i,j)+\mathds{1}_{d'}(i,j)$ for every $i,j$, therefore  	$$\sum_{(r,q) \in S} \beta_{rq} = \left(\sum_{(r,q) \in S} \alpha_{rq}\right) - |S \cap d| + |S \cap d'| \ge \sum_{(r,q) \in S} \alpha_{rq}$$

and the thesis follows.
\end{proof}

\begin{lemm}\label{sumdrepresentation}
	Let $g \in K[x_{11},\ldots,x_{nn}]$ be a monomial and $M(g)=(\alpha_{ij})$ be its exponent matrix. Let $d$ be a diagonal, and let $g'$ be a monomial in a $d$-representation of $g$, having exponent matrix $M(g')=(\beta_{ij})$.
	
	Then for every set $S$ having minimal intersection with $d$ we have
	
	$$\sum_{(r,q) \in S} \alpha_{rq} \le \sum_{(r,q) \in S} \beta_{rq}.$$
	
	Furthermore, equality holds if $S$ has minimal intersection with every diagonal.
\end{lemm}

\begin{proof}
	Since the matrix $M(g')$ is obtained from $M(g)$ by a chain of swaps, the thesis is a direct consequence of Lemma \ref{sumswap}, noticing that if $S$ has minimal intersection with every diagonal, equality holds at each step of the chain of swaps in Lemma \ref{sumswap}.
\end{proof}

\begin{prop}\label{specialin}
	Let $g=x_{11}^{\alpha_{11}}\ldots x_{nn}^{\alpha_{nn}} \in K[x_{11},\ldots,x_{nn}]$ be a monomial, and assume that there exists a diagonal $d \in \D$ which is $k$-special for $g$.
	
	Then $g \in (\mm^{[k]},f)$.
\end{prop}

\begin{proof}
	Let $M(g)=(\alpha_{ij})$ be the exponent matrix of $g$, and consider a $d$-representation $g \equiv g_1+\ldots+g_h \pmod{f}$ of $g$. Take a monomial $g'$ appearing in this $d$-representation, and let $M(g')=(\beta_{ij})$ be its exponent matrix. By definition of $d$-representation we have $\beta_{rq}=0$ for some $(r,q) \in d$. Since $d$ is $k$-special for $g$, there exists a set $S \subseteq \{1,\ldots,n\} \times \{1,\ldots,n\}$ containing $(r,q)$ and having minimal intersection with $d$, such that $$\sum_{(i,j) \in S} \alpha_{ij} \ge (|S|-1)(k-1)+1.$$
	
	By Lemma \ref{sumdrepresentation} we have $$\sum_{(i,j) \in S} \beta_{ij} \ge \sum_{(i,j) \in S} \alpha_{ij} \ge (|S|-1)(k-1)+1,$$
	
	and since $\beta_{rq}=0$ there are at most $|S|-1$ non-zero elements among the summands $\beta_{ij}$ with $(i,j) \in S$: therefore, at least one of them has to be at least $k$, and thus $g' \in \mm^{[k]}$. Since this holds for all monomials $g'$ in the $d$-representation of $g$, we obtain $g \in (\mm^{[k]},f)$.
\end{proof}

The previous Proposition is the key result of this work. In the following, we prove that monomials of certain degrees (sometimes under specific conditions) have at least one $k$-special diagonal; therefore, Proposition \ref{specialin} ensures that those monomials belong to $(\mm^{[k]},f)$, and in turn this will give a bound on $c^\mm(\mm)$.

In all the following results, if $g=x_{11}^{\alpha_{11}}\ldots x_{nn}^{\alpha_{nn}}$ is a monomial, we will always assume $\alpha_{ij} \le k-1$ for every $i,j$, since otherwise it is obvious that $g \in (\mm^{[k]},f)$.

\begin{thm} \label{partial result square}
	Let $g=x_{11}^{\alpha_{11}}\ldots x_{nn}^{\alpha_{nn}}$ be a monomial of degree $\left(n^2-  \left \lfloor{\dfrac{n}{2}} \right\rfloor-1 \right)(k-1)+1$. Then $g\in(\mm^{[k]},f)$. 
	
	In particular, $\nu_k(\mm)\leq\left(n^2-  \left \lfloor{\dfrac{n}{2}} \right\rfloor-1 \right)(k-1)$ and $c^\mm(\mm)\leq n^2-  \left \lfloor{\dfrac{n}{2}} \right\rfloor-1$.
\end{thm}

\begin{proof}
	 We aim to prove that there exists a diagonal $d \in \D$ which is $k$-special for $g$. The thesis will then follow by Proposition \ref{specialin}.
	
	First, the sum of elements on each row of $M(g)$ is at most $n(k-1)$. Let $t$ be the number of rows such that the sum of the elements in each row is at least $(n-1)(k-1)+1$. If $t \le \lceil{\dfrac{n}{2}}\rceil-1 $, then the sum of the elements in each of the remaining $n-t$ rows is at most $(n-1)(k-1)$. Therefore, summing all rows we obtain
	\[ \deg g \leq tn(k-1) + (n-t)(n-1)(k-1)= (k-1)(n^2-n+t) \le (k-1)(n^2- \lfloor{\dfrac{n}{2}}\rfloor -1), \]
	that is a contradiction. Then there are at least $\lceil{\dfrac{n}{2}}\rceil $ rows such that the sum of the elements in each row is at least equal to $(n-1)(k-1)+1$, and the same hold for columns.
	To conclude this proof we extract a $k$-special diagonal from this set of rows and columns, in the following way. 
	
	If $n=2h$ is even we consider $h$ rows and $h$ columns whose sum is at least $(n-1)(k-1)+1$. We pick an element in each row so that none of these is in one of the $h$ columns; we can do that because there are $n$ elements in each row, and only $\frac{n}{2}=h$ belong also to the chosen columns. In an analogous way we pick one element in each column such that none of these belongs to one of the $h$ chosen rows. By construction we have $h+h=n$ elements such that any two of them do not share the same row or column, thus forming a diagonal, and this diagonal is $k$-special since each element belongs to a row or column (thus containing $n$ elements) such that the sum of its elements is at least $(n-1)(k-1)+1$. 
	
	If $n=2h+1$, we consider $h+1$ rows and $h+1$ columns whose sum is at least $(n-1)(k-1)+1$. We pick an element that is in the intersection of one of these rows and one of theses columns. We delete the corresponding row and column and consider the remaining $2h \times 2h$ minor, and we proceed in the same way as in the even case, chosing $2h$ elements that do not share any row or column. Also in this case we have $2h+1=n$ elements that form a diagonal, which is $k$-special for $g$ because by construction each element lies in a row or column whose sum is at least $(n-1)(k-1)+1$.	
\end{proof}

The bound provided in Theorem \ref{partial result square} is most likely not sharp. In fact, for $n=3,4$ we give the exact formula for $c^\mm(\mm)$. 

\subsection{Case $t=m=n=3$}
\begin{thm}\label{sol3}
	Let $g = x_{11}^{\alpha_{11}} \ldots x_{33}^{\alpha_{33}} \in K[x_{11},\ldots,x_{33}]$ be a monomial of degree $6k-5$. Then $g \in (\mm^{[k]},f)$.
\end{thm}

\begin{proof}
	
	By Proposition \ref{repexists} for every diagonal $d$ there exists a $d$-representation of $g$ of the form $g \equiv g_1+\ldots+g_h \pmod{f}$, and by construction $\deg g' = \deg g = 6k-5$ for every monomial $g'$ appearing in this $d$-representation. Clearly, if every monomial appearing in this $d$-representation belongs to $(\mm^{[k]},f)$, then so does $g$ as well. Therefore, it suffices to prove the thesis under the assumption that one of the $\alpha_{ij}$ is equal to zero. 
	
	Since $\deg g=\alpha_{11}+\ldots+\alpha_{33} = 6k-5$, there is at least one row and one column of $M(g)$ such that the sum of its elements is at least $2k-1$. Assume then that $\alpha_{11}+\alpha_{12}+\alpha_{13} \ge 2k-1$ and $\alpha_{11}+\alpha_{21}+\alpha_{31} \ge 2k-1$. Clearly, none of these elements can be zero; thus we can assume that $\alpha_{33}=0$ (the other cases are identical).

	Consider the elements of the second column of $M(g)$. If $\alpha_{12}+\alpha_{22}+\alpha_{32} \ge 2k-1$ then the diagonal $d=\{(3,1),(2,2),(1,3)\}$ is $k$-special (since each element of $d$ belongs to a row or column which sum of elements is at least $2k-1$), hence $g \in (\mm^{[k]},f)$ by Proposition \ref{specialin}. 
	
	On the other hand, if $\alpha_{12}+\alpha_{22}+\alpha_{32} < 2k-1$, then $\alpha_{12}+\alpha_{22}+\alpha_{32}+\alpha_{31} \le 3k-3$, implying that (since $\alpha_{33}=0$) $$\alpha_{11}+\alpha_{13}+\alpha_{21}+\alpha_{23} = 6k-5 - \alpha_{12}-\alpha_{22}-\alpha_{32}-\alpha_{31} \ge 3k-2.$$
	Then, the diagonal $d=\{(1,2),(2,3),(3,1)\}$ is $k$-special for $g'$ because:
	\begin{enumerate}
		\item The first row is a set having minimal intersection with $d$ and containing $(1,2)$, and such that the sum of its elements is at least $2k-1$.
		\item Thr minor $S=\{(1,1),(1,3),(2,1),(2,3)\}$ of size $2=n-1$ is such that $|S \cap d|=1 = n-2$, then by Remark \ref{minimalsets} it has minimal intersection with $d$ and contains $(2,3)$, and 
		$$ \sum_{(r,q) \in S} \alpha_{rq} = \alpha_{11}+\alpha_{13}+\alpha_{21}+\alpha_{23} \ge 3k-2=(|S|-1)(k-1)+1.$$
		\item The first column is a set having minimal intersection with $d$ and containing $(3,1)$, and such that the sum of its elements is at least $2k-1$.
	\end{enumerate}

	Then $g \in (\mm^{[k]},f)$ by Proposition \ref{specialin}. 

\end{proof}

\subsection{Case $t=m=n=4$}
In the case $t=m=n=4$ we prove that every monomial of degree $12k-11$ belongs to $(\mm^{[k]},f)$ (Theorem \ref{sol4}). However, since we need to consider more cases, we break down the main proof in small steps.

\begin{lemm}\label{223113}
	Let $g \in K[x_{11},\ldots,x_{44}]$ be a monomial, and let $M=(\alpha_{ij})$ be its exponent matrix.
	If there are at least two rows and two columns, or one row and three columns (or one column and three rows) such that the sum of the elements on each of these rows and columns is at least $3k-2$, then $g \in (\mm^{[k]},f)$.
\end{lemm}
\begin{proof}
	Assume that there are at least two rows and two columns which sum of elements is at least $3k-2$ (the other cases are similar). We can assume without loss of generality that the first two rows and the first two columns are the ones such that the sum of its elements is at least $3k-2$.
	Then the diagonal $d=\{(1,4),(2,3),(3,2),(4,1)\}$ is $k$-special since every element of $d$ belongs to either the first two rows or the first two columns, therefore $g \in (\mm^{[k]},f)$ by Proposition \ref{specialin}.
\end{proof}

\begin{lemm}\label{2112}
	Let $g \in K[x_{11},\ldots,x_{44}]$ be a monomial of degree $12k-11$, and let $M(g)=(\alpha_{ij})$ be its exponent matrix. Assume that at least one of the $\alpha_{ij}$ is zero, and that there are exactly two rows and one column (or one row and two columns) of $M$ which sum is at least $3k-2$. Then $g \in (\mm^{[k]},f)$.
\end{lemm}

\begin{proof}
	Up to switching some rows and columns (this action does not change the determinant and is equivalent to renaming our variables), we can make the following assumptions:
	\begin{enumerate}[(a)]
		\item $\alpha_{11}+\ldots+\alpha_{14} \ge 3k-2$, $\alpha_{21}+\ldots+\alpha_{24}\ge 3k-2$, $\alpha_{11}+\ldots+\alpha_{41} \ge 3k-2$;
		\item $\alpha_{31}+\ldots+\alpha_{34} \le 3k-3$, $\alpha_{12}+\ldots+\alpha_{42}\le 3k-3$, $\alpha_{13}+\ldots+\alpha_{43} \le 3k-3$;
		\item $\alpha_{44}=0$.
	\end{enumerate}
	Consider the diagonal $d=\{(1,3),(2,2),(3,4),(4,1)\}$. We have:
	\begin{enumerate}
		\item From (a) we have that the first row is a set containing $(1,3)$ and having minimal intersection with $d$ (see Remark \ref{minimalsets}) which sum is at least $3k-2$.
		\item From (a) again we have that the second row is a set containing $(2,2)$ and having minimal intersection with $d$ (see Remark \ref{minimalsets}) which sum is at least $3k-2$.
		
		\item Consider the set $S=\{(1,1),(1,4),(2,1),(2,4),(3,1),(3,4)\}$ highlighted in red in $M(g)$
		$$M(g)=\begin{pmatrix}
			\color{red}{\alpha_{11}} & \alpha_{12} & \alpha_{13} & \color{red}{\alpha_{14}}\\
			\color{red}{\alpha_{21}} & \alpha_{22} & \alpha_{23} & \color{red}{\alpha_{24}} \\
			\color{red}{\alpha_{31}} & \alpha_{32} & \alpha_{33} & \color{red}{\alpha_{34}} \\
			\alpha_{41} & \alpha_{42} & \alpha_{43} & 0
		\end{pmatrix}.$$
		
		This set contains $(3,4)$, $|S \cap d|=1$ and intersects every diagonal at least once (since every diagonal must contain an element from the first and last column, and cannot contain both $(4,1)$ and $(4,4)$): thus $S$ has minimal intersection with $d$. 
		
		Moreover, from (b) we have $\alpha_{12}+\ldots+\alpha_{42}\le 3k-3$ and $\alpha_{13}+\ldots+\alpha_{43} \le 3k-3$, then  
		
		$$\sum_{(r,q) \in S} \alpha_{rq} = 12k-11 - (\alpha_{12}+\ldots+\alpha_{42})-(\alpha_{13}+\ldots+\alpha_{43})-\alpha_{41} \ge $$ $$\ge 12k-11 - (3k-3)-(3k-3)-(k-1)=5k-4=(|S|-1)(k-1)+1.$$
		\item From (a) we have that the first column is a set containing $(4,1)$ and having minimal intersection with $d$ which sum is at least $3k-2$.
	\end{enumerate}
	Then $d$ is $k$-special for $g$, and thus $g \in (\mm^{[k]},f)$ by Proposition \ref{specialin}.
\end{proof}

\begin{lemm}\label{11}
	Let $g \in K[x_{11},\ldots,x_{44}]$ be a monomial of degree $12k-11$, and let $M(g)=(\alpha_{ij})$ be its exponent matrix. Assume that at least one of the $\alpha_{ij}$ is zero, and that there is exactly one row and one column of $M$ which sum of elements is at least $3k-2$. Then $g \in (\mm^{[k]},f)$.
\end{lemm}

\begin{proof}
	Up to a permutation of rows and columns, we can make the following assumptions.
	\begin{enumerate}
		\item[R1)] $\alpha_{11}+\ldots+\alpha_{14}\geq 3(k-1)+1$;
		\item[Ri)] $\alpha_{i1}+\ldots+\alpha_{i4}\leq 3(k-1)$ for $i=2,3,4$;
		\item[C1)] $\alpha_{11}+\ldots+\alpha_{41}\geq 3(k-1)+1$;
		\item[Ci)] $\alpha_{1i}+\ldots+\alpha_{4i}\leq 3(k-1)$ for $i=2,3,4$;
		\item[Z)] $\alpha_{44}=0$.
	\end{enumerate}
	
	Consider the diagonal $d=\{(1,4),(2,1),(3,3),(4,2)\}$. If $d$ is $k$-special for $g$, then $g \in (\mm^{[k]},f)$ by Proposition \ref{specialin}. 
	
	Assume then that $d$ is not $k$-special for $g$. Let $g \equiv g_1+\ldots+g_h \pmod{f}$ be a $d$-representation of $g$, and let $g'=x_{11}^{\beta_{11}}\ldots x_{nn}^{\beta_{nn}}$ be a monomial appearing in this $d$-representation. By definition of $d$-representation there must exist an index $(i,j) \in d$ such that $\beta_{ij}=0$. We have four possible cases depending on which $\beta_{ij}$ is zero; we want to prove that $g' \in (\mm^{[k]},f)$ in all these cases. This will imply $g \in (\mm^{[k]},f)$. 
	\begin{enumerate}
		\item If $\beta_{14}=0$, since rows and columns have minimal intersection with every diagonal by Remark \ref{minimalsets}, by Lemma \ref{sumdrepresentation} we have $\beta_{11}+\ldots+\beta_{14} = \alpha_{11}+\ldots+\alpha_{14} \ge 3(k-1)+1$, that is $\beta_{11}+\beta_{12}+\beta_{13} \ge 3(k-1)+1$. Then at least one among $\beta_{11},\beta_{12},\beta_{13}$ is at least $k$, and $g' \in (\mm^{[k]},f)$. 
		\item If $\beta_{21}=0$, arguing as in the previous case we have $\beta_{11}+\ldots+\beta_{41} = \alpha_{11}+\ldots+\alpha_{41} \ge 3(k-1)+1$, that is $\beta_{11}+\beta_{31}+\beta_{41} \ge 3(k-1)+1$. Then at least one among $\beta_{11},\beta_{31},\beta_{41}$ is at least $k$, and $g' \in (\mm^{[k]},f)$. 
		\item If $\beta_{42}=0$, consider the set $S=\{(1,1),(1,2),(1,3),(4,1),(4,2),(4,3)\}$. Since $|S \cap d|=1$ and $S$ intersects every other diagonal (since every diagonal must contain an element from the first and last row, and cannot contain both $(1,4)$ and $(4,4)$), $S$ has minimal intersection with $d$. Then by Lemma \ref{sumdrepresentation} we have $$\sum_{(r,q) \in S} \beta_{rq} \ge \sum_{(r,q) \in S} \alpha_{rq} = 12k-11 - (\alpha_{21}+\ldots+\alpha_{24})-(\alpha_{31}+\ldots+\alpha_{34}) - \alpha_{14},$$ 
		
		hence by R2), R3) and $\alpha_{14} \le k-1$ we deduce $$\sum_{(r,q) \in S} \beta_{rq} \ge 12k-11 - 3(k-1)-3(k-1)-(k-1)\ge 5(k-1)+1,$$ 
		 hence $\beta_{11}+\beta_{12}+\beta_{13}+\beta_{41}+\beta_{43} \ge 5(k-1)+1$ and at least one $\beta_{ij}$ in that sum is at least $k$. Then $g' \in (\mm^{[k]},f)$. 
		\item If $\beta_{33}=0$, consider the diagonal $d'=\{(1,4),(2,3),(3,1),(4,2)\}$. We verify that $d'$ is $k$-special for $g'$ by checking the $k$-special property in each index $(i,j) \in d'$:
		\begin{itemize}
			\item If $(i,j)=(1,4)$, the first row contains $(1,4)$, has minimal intersection with both $d$ and $d'$ by Remark \ref{minimalsets}, thus by Lemma \ref{sumdrepresentation} and R1) $\beta_{11}+\ldots+\beta_{14} \ge \alpha_{11}+\ldots+\alpha_{14} \ge 3(k-1)+1$, and the condition is satisfied; 
			\item If $(i,j)=(3,1)$, the first column contains $(3,1)$, has minimal intersection with both $d$ and $d'$ by Remark \ref{minimalsets}, thus by Lemma \ref{sumdrepresentation} and R1) $\beta_{11}+\ldots+\beta_{41} \ge \alpha_{11}+\ldots+\alpha_{41} \ge 3(k-1)+1$, and the condition is satisfied;
			\item If $(i,j)=(4,2)$, consider the set $S:=\{(1,1),(1,2),(1,3),(4,1),(4,2),(4,3)\}$.
			From R2), R3) and Z) we deduce $\alpha_{11}+\ldots+\alpha_{14}+\alpha_{41}+\ldots+\alpha_{43} \ge 6(k-1)+1$, and since $\alpha_{14} \le k-1$ we obtain $$\sum_{(r,q) \in S} \alpha_{rq}= \alpha_{11}+\ldots+\alpha_{13}+\alpha_{41}+\ldots+\alpha_{43} \ge 5(k-1)+1.$$ 
			Furthermore, since $S$ intersects every diagonal in at least one element, and $|S \cap d|=1$, $S$ has minimal intersection with $d$, and thus Lemma \ref{sumdrepresentation} yields
			$$\sum_{(r,q) \in S} \beta_{rq} \ge \sum_{(r,q) \in S} \alpha_{rq} \ge 5(k-1)+1.$$ 
			Since $S$ has minimal intersection with $d'$ and contains $(4,2)$, the condition is satisfied;
			\item  If $(i,j)=(2,3)$, consider the set $S':=\{(1,1),(2,1),(4,1),(1,3),(2,3),(4,3)\}$. Since rows and columns have minimal intersection with every diagonal, by Lemma \ref{sumdrepresentation} and C2), C4) we have $\beta_{12}+\ldots+\beta_{42} = \alpha_{12}+\ldots + \alpha_{42} \le 3(k-1)$ and $\beta_{14}+\ldots+\beta_{44} = \alpha_{14}+\ldots + \alpha_{44} \le 3(k-1)$. Then, since $\beta_{33}=0$ and $\beta_{31} \le k-1$ we obtain
			$$\sum_{(r,q) \in S'} \beta_{rq} = 12(k-1)+1 - (\beta_{12}+\ldots+\beta_{42})-(\beta_{14}+\ldots+\beta_{44})-\beta_{31} \ge $$ $$\ge 12(k-1)+1-3(k-1)-3(k-1)-(k-1)=5(k-1)+1.$$
			Since $S'$ intersects every diagonal at least once and $|S' \cap d'|=1$, $S'$ has minimal intersection with $d'$ (and contains $(2,3)$), thus the condition is satisfied.
		\end{itemize}
		Then $d'$ is $k$-special for $g'$, and $g' \in (\mm^{[k]},f)$ by Proposition \ref{specialin}.
	\end{enumerate}
\end{proof}

\begin{thm}\label{sol4}
	Let $g \in K[x_{11},\ldots,x_{44}]$ be a monomial of degree $12k-11$. 
	
	Then $g \in (\mm^{[k]},f)$.
\end{thm}

\begin{proof}
	Let $d \in \D$ be a diagonal, and consider a $d$-representation of $g$ of the form $g \equiv g_1+\ldots+g_h \pmod{f}$. Let $g'$ be a monomial appearing in this $d$-representation of $g$. By definition of $d$-representation, at least one entry of the exponent matrix $M(g')$ is equal to zero.
	Since ${\rm deg}(g')=12k-11$, there is at least one row and one column of $M(g')$ which sum is at least $3k-2$. Depending on their number, we have the following cases:
	\begin{enumerate}
		\item If there are at least three rows and one column (or three columns and one row) which sum is at least $3k-2$, then $g' \in (\mm^{[k]},f)$ by Lemma \ref{223113};
		\item If there are at least two rows and columns of $M'$ which sum is at least $3k-2$, then $g' \in (\mm^{[k]},f)$ by Lemma \ref{223113};
		\item If there are two rows and one column (or one row and two columns) of $M'$ which sum is at least $3k-2$, then $g' \in (\mm^{[k]},f)$ by Lemma \ref{2112};
		\item If there is one rows and one column of $M'$ which sum is at least $3k-2$, then $g' \in (\mm^{[k]},f)$ by Lemma \ref{11};
	\end{enumerate}
	Then $g' \in (\mm^{[k]},f)$ for all $g'$ appearing in the $d$-representation of $g$, and thus $g \in (\mm^{[k]},f)$.
\end{proof}

Now we can give the exact value of $c^\mm(\mm)$ for $n=3,4$.

\begin{cor}\label{cm34}
	In the context of problem \ref{prob} with $n=m=t$, if $n=3,4$ then $c^{\mm}(\mm) = n(n-1)$.
\end{cor}

\begin{proof}
	By Theorem \ref{lowerbound} we have	$n(n-1) = {\rm fpt}(\mm) \leq c^{\mm}(\mm)$. On the other hand, if $n=3,4$ by Theorem \ref{sol3} and \ref{sol4} with $k=p^e$ we obtain $\nu_{p^e}(\mm) \le n(n-1)(p^e-1)$. Then $\displaystyle c^{\mm}(\mm)=\lim_{e\rightarrow\infty}\frac{\nu_{p^e}(\mm)}{p^e}\leq n(n-1)$, and the thesis follows.
\end{proof}

In light of this last result we state the following conjecture. 
\begin{conjecture}\label{squareconj}
	If $n=m=t$ then $c^\mm(\mm)=n(n-1)$.
\end{conjecture}

\section{Proof of main results} 
In this Section we build on the results obtained in the case $t=m=n$ to attack the more general case $t=m \le n$ of Problem \ref{altprob} and prove the main Theorems stated in the introduction. Since the ideal $I=I_t(X)$ is not principal in this case, instead of the combinatorial machinery introduced in the case $t=m=n$ used to work modulo $f={\rm det}(X)$, we use a counting argument to find a suitable $t \times t$ minor $T$, and then work modulo ${\rm det}(T)$, leveraging the results obtained in the previous Section.

\begin{prop}\label{t=m}
Let $g \in K[x_{11},\ldots,x_{nm}]$ be a monomial.
\begin{enumerate}
	\item If $\deg g \ge n\left(t - \frac{1}{2} \right)(k-1)+1$, then $g \in (\mm^{[k]},I)$;
	\item If $t=m=3,4$ and $\deg g \ge n(t-1)(k-1)+1$, then $g \in (\mm^{[k]},I)$.
\end{enumerate}
\end{prop}

\begin{proof}

Consider the $n \times m$ exponent matrix $M(g)$. Since $t=m$, there are exactly $\binom{n}{t}$ different $t \times t$  minors in $M(g)$ (obtained by choosing $t$ rows among $n$), which we will denote by $T_i$. For a given $t \times t$ minor $T_i$, let $\sigma_i$ be the sum of all elements of $T_i$. Since every element of $M(g)$ belongs to exactly $\binom{n-1}{t-1}$ minors (the ones including its row among the $t$ chosen), we have 
$$\sum_{i=1}^{\binom{n}{t}} \sigma_i = \binom{n-1}{t-1}\deg g.$$

Thus, denoting by $\tilde{\sigma}$ the arithmetic mean of the $\sigma_i$-s, we obtain that

$$\tilde{\sigma}= \frac{\binom{n-1}{t-1}\deg g}{\binom{n}{t}}=\frac{t}{n}\deg g.$$ 
By definition of $\tilde{\sigma}$ there exists an index $l$ such that $\sigma_l \ge \tilde{\sigma} \ge \frac{t}{n}\deg g$. By definition, $\sigma_l$ is the sum of all elements in the associated minor $T_l$ of the exponent matrix $M(g)$. Let $\tilde{X}$ be the $t \times t$ minor of the matrix of variables $X$ associated to the exponents appearing in $T_l$, let $\tilde{S}=K[\tilde{X}]$ be the associated polynomial ring, $\tilde{\mm}$ be its maximal homogeneous ideal and let $\tilde{f} := {\rm det}(\tilde{X})$. Notice that $\tilde{f} \in I=I_t(X)$.

Let $\tilde{g}$ be the monomial of $\tilde{S}$ such that $M(\tilde{g}) = T_l$.
\begin{enumerate}
	\item If $\deg g \ge n\left(t - \frac{1}{2} \right)(k-1)+1$, we obtain $\sigma_l \ge \left(t^2 - \frac{t}{2} \right)(k-1)+\frac{1}{n}.$ Since $t^2-\frac{t}{2} \ge t^2-\left \lfloor \frac{t}{2} \right\rfloor - 1$ and $\sigma_l$ is an integer we actually obtain $$\sigma_l \ge \left(t^2-\left \lfloor \frac{t}{2} \right\rfloor - 1 \right)(k-1)+1.$$
	
	By construction, $\deg \tilde{g}=\sigma_l$, and by Theorem \ref{partial result square} we deduce $\tilde{g} \in (\tilde{\mm}^{[k]},\tilde{f}).$ Then, since $\tilde{g}$ divides $g$ and $f \in I$ we obtain $g \in (\mm^{[k]},I)$.
	\item If $t=m=3,4$ and $\deg g \ge n(t-1)(k-1)+1$ then $\sigma_l \ge t(t-1)(k-1)+1$. By construction $\deg \tilde{g}=\sigma_l$, hence Theorems \ref{sol3} and \ref{sol4} imply that $\tilde{g} \in (\tilde{\mm}^{[k]},\tilde{f}),$ and thus $g \in (\mm^{[k]},I)$.
\end{enumerate}
\end{proof}

Now we prove the main results.

\setcounter{thm}{3}
\begin{thm}
    Let $R=S/I_t(X)$, where $X$ is a $n\times m$ matrix of indeterminates and $t=m\leq n$. Then for $n\geq 3$ we have
    \[  c^{\mm}(\mm)\leq  n\left(t-\frac{1}{2}\right)  .
  \]
  
\end{thm}

\begin{proof}
By Proposition \ref{t=m} (1) with $k=p^e$ we obtain $\nu_{p^e}(\mm) \le n\left(t - \frac{1}{2} \right)(p^e-1),$ and thus  $$\displaystyle c^{\mm}(\mm)=\lim_{e\rightarrow\infty}\frac{\nu_{p^e}(\mm)}{p^e}\leq n\left(t-\frac{1}{2}\right).$$
\end{proof}

\begin{thm}
    Let $R=S/I_t(X)$, where $X$ is a $n\times m$ matrix of indeterminates and $t=m\leq n$. Then for $t=3,4$ we have  
    \[    c^{\mm}(\mm)=n(t-1). \]
\end{thm}

\begin{proof}
Proposition \ref{t=m} part (2) applied with $k=p^e$ proves that 
\[\nu_{p^e}(\mm) \le n(t - 1)(p^e-1) \] and thus  $\displaystyle c^{\mm}(\mm)=\lim_{e\rightarrow\infty}\frac{\nu_{p^e}(\mm)}{p^e}\leq n(t-1).$ The thesis follows from the lower bound in Theorem \ref{lowerbound} with the negative $a$-invariant, that is $n(t-1) = -a(R)\le c^\mm(\mm) \le n(t-1)$.
\end{proof}

The argument outlined in these proofs actually shows that the more general Conjecture \ref{conj F-threshold} for the case $t=m\le n$ is actually a consequence of Conjecture \ref{squareconj} for the case $t=m=n$, and any progress in solving the latter would give an analogous result on the former.

\section*{Acknowledgements} 
         We are grateful to the organizers of the P.R.A.G.MAT.I.C. research school held in Catania in June 2023 that gave us the opportunity to work together. We thank Luis Nuñez-Betancourt, Eamon Quinlan-Gallego and Alessio Sammartano for the helpful conversations and their careful proof-reading. The fourth author is partially supported by MATRICS grant MTR/2021/000879.
        
	\small
	\bibliographystyle{alpha}
	\bibliography{References}
	
\end{document}